\documentclass{article}

\usepackage{amsfonts}
\usepackage{amsthm}
\usepackage{amsmath}
\usepackage{amssymb}
\usepackage{mathtools} %for prescript
\usepackage{authblk} %for affil
\usepackage{xcolor}
\usepackage{soul} %for strikethrough
\usepackage{enumerate}
\usepackage{hyperref} %for url

\usepackage[top=1in, bottom=1in, left=1in, right=1in]{geometry}

\newtheorem{thm}{Theorem}[section]
\newtheorem{lem}{Lemma}[section]
\newtheorem{propn}{Proposition}[section]

\theoremstyle{definition}
\newtheorem{defn}{Definition}[section]

\theoremstyle{remark}
\newtheorem{remark}{Remark}[section]

\DeclareMathOperator*{\esssup}{ess\,sup}

\begin{document}

%\begin{frontmatter}

\title{Unique axiomatisation of fractional integral operators on the real line}

%%this line removes the date, but space is still left for it;
%if used, remove the \vspace{-1cm}
\date{}

%this gives the date in the form Mon 30 Jan 2012, 8:57pm;
%if used, retain the \vspace{-1cm}
%\date{\shortdayofweekname{\day}{\month}{\year}{ }\mydate\today}

\author[1,2]{Ceren \"Ozden}
\author[1]{Arran Fernandez\thanks{Corresponding author. Email: \texttt{arran.fernandez@emu.edu.tr}}}

\affil[1]{{\small Department of Mathematics, Eastern Mediterranean University, 99628 Famagusta, Northern Cyprus, via Mersin-10, Turkey}}
\affil[2]{{\small Department of Mathematics, Istanbul Technical University, Katar Caddesi, 34469 Istanbul, Turkey}}

% Latex won't make the title unless told:

\maketitle

%%to remove the space left for date, use:
%\vspace{-1cm}

\begin{abstract}
The Cartwright--McMullen theorem (1978) establishes the uniqueness of the Riemann--Liouville family of fractional integral operators on a compact interval $[0,1]$ under a natural set of axioms: inclusion of the classical integral, a semigroup property, positivity, and continuity. By unpacking the proof into its constituent steps and understanding its structure clearly, we provide improvements and extensions of the original result. Firstly, by using some properties of topological groups, we demonstrate that the positivity axiom can be removed with almost no effect on the result. Secondly, we consider a version of the theorem on the whole real line $\mathbb{R}$, with constant of integration $-\infty$ instead of $0$, which was mentioned without proof in the 1978 paper. The theorem can be extended to this setting, but it is not trivial to do so: the theory of Fr\'echet spaces must be used, and we introduced a new shift-commutativity axiom to get the density result that we need to extend the proof to spaces of functions and distributions on $\mathbb{R}$ with left-bounded support.
\end{abstract}

\section{Introduction}

The idea of extending differentiation and integration beyond integer orders is almost as old as calculus itself, being explicitly referenced in the correspondence between Leibniz and l'H\^opital, in a letter dated 1695. While no precise definition was created at that time, this correspondence shows that the idea of fractional calculus existed from the earliest stages of the development of calculus.

During the nineteenth century, some clear definitions began to emerge for derivatives and integrals of non-integer orders, but it seems that there were always multiple competing definitions. For example, Lacroix started from a very natural formula for fractional derivatives of power functions,
\[
\frac{\mathrm{d}^{n}}{\mathrm{d}x^n}\big(x^m\big)=\frac{\Gamma(m+1)}{\Gamma(m-n+1)}x^{m-n},
\]
and Liouville started from a very natural formula for fractional derivatives of exponential functions,
\[
\frac{\mathrm{d}^{n}}{\mathrm{d}x^n}\big(e^{ax}\big)=a^ne^{ax},
\]
to create two parallel theories that are not consistent with one another. Nowadays, those two parallel theories are both seen as particular cases of the so-called Riemann--Liouville fractional calculus, which encompasses both earlier constructions within a single definition. The Riemann--Liouville fractional integral with constant of integration $c$ is given by
\[
\prescript{RL}{c}D^{-\nu}_{x}f(x)=\frac{1}{\Gamma(\nu)}\int_c^x(x-y)^{\nu-1}f(y)\,\mathrm{d}y,\qquad\operatorname{Re}\nu>0,\quad x>c.
\]
Some straightforward integral calculations show that putting $c=0$ recovers Lacroix's formula,
\[
\prescript{RL}{0}D^{-\nu}_{x}\big(x^m\big)=\frac{\Gamma(m+1)}{\Gamma(m+\nu+1)}x^{m+\nu},\qquad m>-1,
\]
while putting $c=-\infty$ recovers Liouville's formula,
\[
\prescript{RL}{-\infty}D^{-\nu}_{x}\big(e^{ax}\big)=a^{-\nu}e^{ax},\qquad a>0.
\]
See \cite[Chapter I]{miller-ross} for more information on the history of fractional calculus, including the Lacroix--Liouville discrepancy.

Even today, the theory of fractional calculus is developing in several different directions, reflecting different methods and viewpoints, and so many different fractional derivative and integral operators have been defined that it is impossible to count them all \cite{teodoro-machado-oliveira,hilfer-luchko}. On the other hand, many of these new operators have direct relations back to the classical Riemann--Liouville operators, enabling them to be studied more efficiently \cite{fernandez-fahad}. Thus, it seems intuitively that the Riemann--Liouville fractional calculus -- or at least the fractional integral operator, since even this has several fractional derivatives associated with it \cite{diethelm,luchko:level} -- is somehow fundamental and unique within the rapidly expanding world of fractional-calculus operators. This naturally raises a structural question: under what conditions can a fractional integral operator be uniquely determined?

Such questions are not new. In 1974, the first international conference devoted entirely to fractional calculus was held in New Haven. During the session on open problems chaired by Thomas J. Osler \cite{ross}, John S. Lew posed the following question. If $\{I_\alpha\}_{\alpha\geq0}$ is a family of linear operators acting on $L^1(0, 1)$ or $L^2(0, 1)$ and satisfying the following conditions:
\begin{itemize}
\item $I_0 f = f$ and $I_1 f(x) = \int_{0}^{x} f (u) \,\mathrm{d}u$;
\item the semigroup property $I_\alpha I_\beta = I_{\alpha+\beta}$ holds for all $\alpha, \beta \geq 0$;
\item the map $\alpha\mapsto I_\alpha$ is continuous with respect to a suitable topology on $L$;
\item for each $\alpha\geq0$, $I_{\alpha} f \geq 0$ whenever $f \geq 0$ a.e.
\end{itemize}
then must $\{I_\alpha\}_{\alpha\geq0}$ be the family of Riemann--Liouville fractional integrals -- in other words, are these four conditions enough to determine this family uniquely?

This problem was solved within a few years: Donald I. Cartwright and John R. McMullen proved that the answer is yes, in a two-page paper \cite{cartwright-mcmullen} published in 1978. Their theorem (stated precisely as Theorem \ref{Thm:CMorig} below) shows that the family of Riemann--Liouville fractional integral operators is unique within the space of bounded linear operators on function spaces such as $L^p[0,1]$ or $C[0,1]$ satisfying the given set of conditions. Despite the power and value of this result, and its significance for the entire structure of fractional calculus, it has remained almost unknown within the research community, with fewer than 20 citations at the time of writing.

A few recent papers from Daniel Cao Labora and Marc Jornet have played a role in reigniting interest in the Cartwright--McMullen theorem. They extended the original result to fractional integrals with respect to monotonic functions \cite{caolabora} and to weighted fractional integrals with respect to functions in higher dimensions \cite{jornet}, using transmutation relations \cite{fernandez-fahad} to reduce these settings to the original Riemann--Liouville case. Our work goes in a different direction, seeking to extend the results on the original Riemann--Liouville fractional calculus rather than proving uniqueness of other operator families.

It is worth noting that, although the Cartwright--McMullen paper was only two pages, their proof was much more deep and intricate than it appears at first glance, relying on a non-trivial combination of ideas from functional analysis and convolution theory. Some of the necessary steps and facts in this proof were glossed over or presented in a highly compressed form in the original paper. Thus, section \ref{Sec:CMstructure} below is devoted to detailing the steps of the proof more clearly and thoroughly. We have also depicted the proof structure visually in a diagram (Figure \ref{Fig1}) which may serve as a graphical abstract of this paper.

Our main results lie in extending the Cartwright--McMullen theorem in two ways. Firstly, we notice that the positivity axiom is only used in a very weak sense, and we conjecture that it can be removed entirely. This conjecture is proved using an argument involving continuous homomorphisms between topological groups, and it is the main result of section \ref{Sec:nopos} of this paper. Secondly, we consider spaces of functions on the whole real line $\mathbb{R}$, rather than just on the compact interval $[0,1]$, and extend the axiomatisation result to this setting. This work is done in section \ref{Sec:Lploc}, where the main result is Theorem \ref{Thm:Lploc} proving the unique axiomatisation of Riemann--Liouville fractional integrals on the spaces of functions in $L^p_{\mathrm{loc}}(\mathbb{R})$ or $C^k(\mathbb{R})$ with left-bounded support. Finally, we also extend this result to the space of distributions with left-bounded support, in section \ref{Sec:distrib} with the main result being Theorem \ref{Thm:D'}. Not much work has been done on fractional calculus in spaces of distributions, and here we take inspiration from recent work of Hilfer and Kleiner \cite{hilfer-kleiner,kleiner-hilfer}. Results in spaces such as $L^p_{\mathrm{loc}}(\mathbb{R})$ and $C^k(\mathbb{R})$, and even the space of distributions with left-bounded support, are hinted at by Cartwright and McMullen in the last paragraph of their paper \cite{cartwright-mcmullen}, but never given explicitly, and the proofs are not as trivial as their brief mention there seems to imply. In fact, the set of required axioms is not even the same: our proof on $\mathbb{R}$ requires an extra assumption which is not mentioned in the Cartwright--McMullen paper. As with other more famous results in mathematics, it turns out that, even if a brief note may seem to be enough, in fact a lot more work is needed for a full proof.

\section{Preliminaries} \label{Sec:prelim}

We begin with an assortment of concepts and facts from the literature which will be used in the subsequent work later on in this paper.

\begin{defn}[Fourier and Laplace type convolutions]
The convolution of two functions $f$ and $g$ both defined on $\mathbb{R}$ is the function $f*g$ defined by
\begin{equation} \label{convol:F}
\big(f*g\big)(x)=\int_{-\infty}^{\infty}f(x-y)g(y)\,\mathrm{d}y.
\end{equation}
This formula defines the so-called Fourier-type convolution. If both $f$ and $g$ have support in $[0,\infty)$, then the above integral can be rewritten as follows:
\[
\big(f*g\big)(x)=\int_{0}^{x}f(x-y)g(y)\,\mathrm{d}y.
\]
This formula, in general, defines the so-called Laplace-type convolution. It can also be used as a natural convolution operation on spaces of functions defined on compact intervals such as $[0,1]$.

These names for the two types of convolution arise from their elegant properties in combination with, respectively, the Fourier and Laplace transforms. The difference between the two convolution formulae is small but significant.
\end{defn}

\begin{defn}
It is a standard fact of functional analysis that the function space $C[0,1]$ is a Banach space under the supremum norm $\|\cdot\|_{\infty}$, or indeed under any $L^p$ norm $\|\cdot\|_p$ for $p\geq 1$, and also that the space $L^p[0,1]$ of equivalence classes of functions (where functions are considered to be equivalent if they are equal almost everywhere) is a Banach space under the corresponding $L^p$ norm $\|\cdot\|_p$ defined by integrals over the interval $[0,1]$.

For spaces of functions defined on the whole real line, we do not have a natural normed space structure and it is necessary to use seminorms instead. As discussed in Rudin \cite[Examples 1.44 and 1.46]{rudin}, the function spaces $C(\mathbb{R})$ and $C^{\infty}(\mathbb{R})$ are Fr\'echet spaces under the countable family of seminorms defined by taking the supremum norm of a function (and of its derivatives, in the case of $C^\infty$) on compact intervals such as $[-n,n]$ whose union is $\mathbb{R}$. Similarly, we can define $L^p_{\mathrm{loc}}(\mathbb{R})$ and $C^k(\mathbb{R})$ as Fr\'echet spaces in the following way.

For $1\leq p\leq\infty$, the space $L^p_{\mathrm{loc}}(\mathbb{R})$ of equivalence classes of functions (where functions are considered to be equivalent if they are equal almost everywhere) is equipped with the following seminorms:
\[
q_K(f) := \|f\|_{L^p(K)} = \begin{cases}\displaystyle\left( \int_K |f(x)|^p \, dx \right)^{1/p}\quad&\quad\text{if }p\in[1,\infty),
\\ \\
\displaystyle\operatorname*{ess\,sup}_{x \in K} \big|f(x)\big|\quad&\quad\text{if }p=\infty,
\end{cases}
\]
where $K$ is a compact subset of $\mathbb{R}$, and it suffices to take a countable nested collection of such subsets whose union is $\mathbb{R}$, such as $K_n=[-n,n]$ for $n\in\mathbb{N}$. It is straightforward to verify that:
\begin{itemize}
\item the topology induced by these seminorms is the natural topology of local $L^p$-convergence;
\item if $q_{K}(f)=0$ for all $K$ in the countable family of compact subsets, then $f=0$ as an equivalence class of functions (thus the topology is Hausdorff);
\item the space $L^p_{\mathrm{loc}}(\mathbb{R})$ is complete with respect to this family of seminorms;
\end{itemize}
and thus we have a Fr\'echet space.

For $k\in\mathbb{Z}^+_0$, the space $C^k(\mathbb{R})$ of $k$ times continuously differentiable functions is equipped with the following seminorms:
\[
q_K(f) := \|f\|_{C^k(K)} = \sup_{x\in K}\big|f(x)\big|+\sup_{x\in K}\big|f'(x)\big|+\sup_{x\in K}\big|f''(x)\big|+\cdots+\sup_{x\in K}\big|f^{(k)}(x)\big|,
\]
where $K$ is a compact subset of $\mathbb{R}$, and it suffices to take a countable nested collection of such subsets whose union is $\mathbb{R}$, such as $K_n=[-n,n]$ for $n\in\mathbb{N}$. Again, it is straightforward to verify that:
\begin{itemize}
\item the topology induced by these seminorms is the natural topology of $C^k$-convergence;
\item if $q_{K}(f)=0$ for all $K$ in the countable family of compact subsets, then $f=0$ (thus the topology is Hausdorff);
\item the space $C^k(\mathbb{R})$ is complete with respect to this family of seminorms;
\end{itemize}
and thus we have a Fr\'echet space.

\begin{remark} \label{Rem:LEfrechet}
The following fact will be useful later. If $E$ is a Fr\'echet space, then the space $\mathcal{L}(E)$ of continuous linear mappings from $E$ to itself is also a seminormed space which is Hausdorff \cite[Theorem 12.3]{simon} and sequentially complete \cite[Section 12.3]{simon}.
\end{remark}

\end{defn}

\begin{defn}[Riemann--Liouville fractional integrals]
For every $\alpha>0$, we define the function $h_{\alpha}\in L^1_{\text{loc}}(0,\infty)$ as follows:
\begin{equation} \label{h:fns}
h_{\alpha}(x)=\begin{cases}
\displaystyle\frac{x^{\alpha-1}}{\Gamma(\alpha)},&\qquad x>0;
\\
0,&\qquad x\leq0.
\end{cases}
\end{equation}
These functions form a convolutional semigroup within $L^1_{\text{loc}}(0,\infty)$, which is isomorphic to the semigroup of positive real numbers under addition \cite{fernandez}.

The fractional integral to order $\alpha>0$, with constant of integration equal to $0$, is defined by a Laplace-type convolution $I^{\alpha}f=f*h_{\alpha}$ for suitable functions $f$ defined on $(0,\infty)$, i.e. by the following expression:
\[
\big(I^{\alpha}_{0}f\big)(x)=\frac{1}{\Gamma(\alpha)}\int_{0}^{x}(x-t)^{\alpha-1}f(t)\,\mathrm{d}t.
\]
The largest possible function space for $f$, such that these operators are defined, is the space $L^1_{\text{loc}}[0,\infty)$, as shown in \cite{martinez-sanz-martinez}.

The fractional integral to order $\alpha>0$, with constant of integration equal to $-\infty$, is defined by a Fourier-type convolution $I^{\alpha}f=f*h_{\alpha}$ for suitable functions $f$ defined on $\mathbb{R}$, i.e. by the following expression:
\[
\big(I^{\alpha}_{-\infty}f\big)(x)=\frac{1}{\Gamma(\alpha)}\int_{-\infty}^{x}(x-t)^{\alpha-1}f(t)\,\mathrm{d}t.
\]
One possible function space for $f$, such that these operators are defined, is the space of all functions $f\in L^1_{\text{loc}}(\mathbb{R})$ which have left-bounded support, i.e. such that $\operatorname{supp}(f)\subset[a,\infty)$ for some finite $a\in\mathbb{R}$. This is not the largest possible function space where the fractional integral operators are defined, e.g. because they are also defined for the exponential function which has support $\mathbb{R}$. On the other hand, these operators cannot be defined on the whole of $L^1_{\text{loc}}(\mathbb{R})$, e.g. because the improper integral diverges when $f$ is a non-zero constant function.

The $h_{\alpha}$ functions satisfy $h_\alpha*h_\beta=h_{\alpha+\beta}$ for all $\alpha,\beta>0$, where the convolution here can be either of Laplace type or of Fourier type, since both are equivalent for functions that are supported in $[0,\infty)$. Consequently, we have $I^\alpha_c\circ I^\beta_c = I^{\alpha+\beta}_c$ for all $\alpha,\beta>0$, for either (or any) value of $c$.

Cartwright and McMullen's original result proved uniqueness of the operators $I^{\alpha}_{0}$ under a suitable set of conditions. Our Theorem \ref{Thm:Lploc} and Theorem \ref{Thm:D'} below will prove uniqueness of the operators $I^{\alpha}_{-\infty}$ under a similar set of conditions, in suitable spaces of functions or distributions respectively.
\end{defn}

\section{The structure of Cartwright and McMullen's proof} \label{Sec:CMstructure}

In this section, we state precisely the original theorem of Cartwright and McMullen concerning operators on function spaces on the compact interval $[0,1]$. We sketch an outline of its proof below, and we have also created, in Figure \ref{Fig1}, a visual depiction of the structure of the proof, with all axioms and intermediate facts and steps shown explicitly. Understanding the proof in this structured way was invaluable for us in figuring out how to extend and modify the result, in the later sections of this paper, and we hope that it will also be useful for others in understanding the proof more clearly and continuing future work on such results.

\begin{thm}[Cartwright--McMullen theorem \cite{cartwright-mcmullen}] \label{Thm:CMorig}
Let $E$ be a space of functions on the interval $[0,1]$, either $E=C[0,1]$ or $E=L^p[0,1]$ with $1\leq p<\infty$. There exists a unique family $\{J_{\alpha}\}_{\alpha>0}$ of bounded linear operators on $E$ satisfying the following four axioms:
\begin{enumerate}[(i)]
\item $\displaystyle \big(J_1f\big)(x)=\int_0^xf(t)\,\mathrm{d}t$ for every $f\in E$ and a.e. $x\in[0,1]$.
\item Semigroup property: $J_{\alpha}J_{\beta}=J_{\alpha+\beta}$ for all $\alpha,\beta>0$.
\item Positivity: for every $\alpha>0$, if $f\in E$ and $f\geq0$, then $J_{\alpha}f\geq0$ a.e.
\item Continuity: the map $\alpha\mapsto J_{\alpha}$ is continuous from $(0,\infty)$ into the space $\mathcal{L}(E)$ under some Hausdorff topology weaker than the operator norm topology on this space.
\end{enumerate}
The unique family satisfying these four axioms is the family of Riemann--Liouville fractional integrals with constant of integration $0$, namely $J_{\alpha}f=I^{\alpha}_{0}f=f*h_{\alpha}$ where the convolution is of Laplace type.
\end{thm}

\begin{proof}[Sketch of proof]
The proof involves the following three basic mathematical objects:
\begin{itemize}
\item a larger function space $A=L^1[0,1]$;
\item a smaller function space $E$, taken to be either $L^p[0,1]$ or $C[0,1]$ as stated in the theorem;
\item a family of operators $R_gf=f*g$ for $f\in E$ and $g\in A$ (where the convolution is of Laplace type).
\end{itemize}
The following facts are proved about these.
\begin{enumerate}
\item For every $g\in A$, the operator $R_g$ is linear and bounded from $E$ into itself.
\item The mapping $g\mapsto R_g$ is continuous from $A$ into $\mathcal{L}(E)$.
\item For every non-trivial $g\in A$ (not identically zero on any interval), the operator $R_g$ is injective.
\item The mapping $g\mapsto R_g$ is injective from $A$ into $\mathcal{L}(E)$.
\item For any $T\in\mathcal{L}(E)$, if $T$ commutes with $R_1=J_1=I_0^1$, then $T$ commutes with $R_g$ for every $g\in A$.
\item For any $g\in A$ and $m\in\mathbb{N}$, if $R_g^m=R_1$, then $R_g=\eta\cdot I_0^{1/m}$ for some $\eta\in\mathbb{C}$ satisfying $\eta^m=1$.
\item For any $T\in\mathcal{L}(E)$ and $m\in\mathbb{N}$, if $T^m=R_1$, then $T=\eta\cdot I_0^{1/m}$ for some $\eta\in\mathbb{C}$ satisfying $\eta^m=1$.
\end{enumerate}
Facts 1 and 2 follow from Young's convolution inequality. Facts 3 and 4 and 6 follow from Titchmarsh's convolution theorem. Fact 5 follows from the denseness of polynomials in $L^1[0,1]$ together with fact 2. Finally, to prove fact 7, the already-proven facts 5 and 4 and 6 and 3 are used respectively, in a non-trivial argument that is essentially reproduced in the proof of Lemma \ref{lem:root_without_assuming_Rg} below.

Axiom (i) already gives $J_{\alpha}=I_0^{\alpha}$ for $\alpha=1$. Axiom (ii), fact 7, and axiom (iii) together give $J_{\alpha}=I_0^{\alpha}$ for $\alpha=1/m$, for any $m\in\mathbb{N}$. Axiom (ii) then gives $J_{\alpha}=I_0^{\alpha}$ for all $\alpha\in\mathbb{Q}^+$. Finally, axiom (iv) gives $J_{\alpha}=I_0^{\alpha}$ for all $\alpha\in\mathbb{R}^+$ by a continuity argument.
\end{proof}

\begin{figure}
\centering
\includegraphics[width=\linewidth]{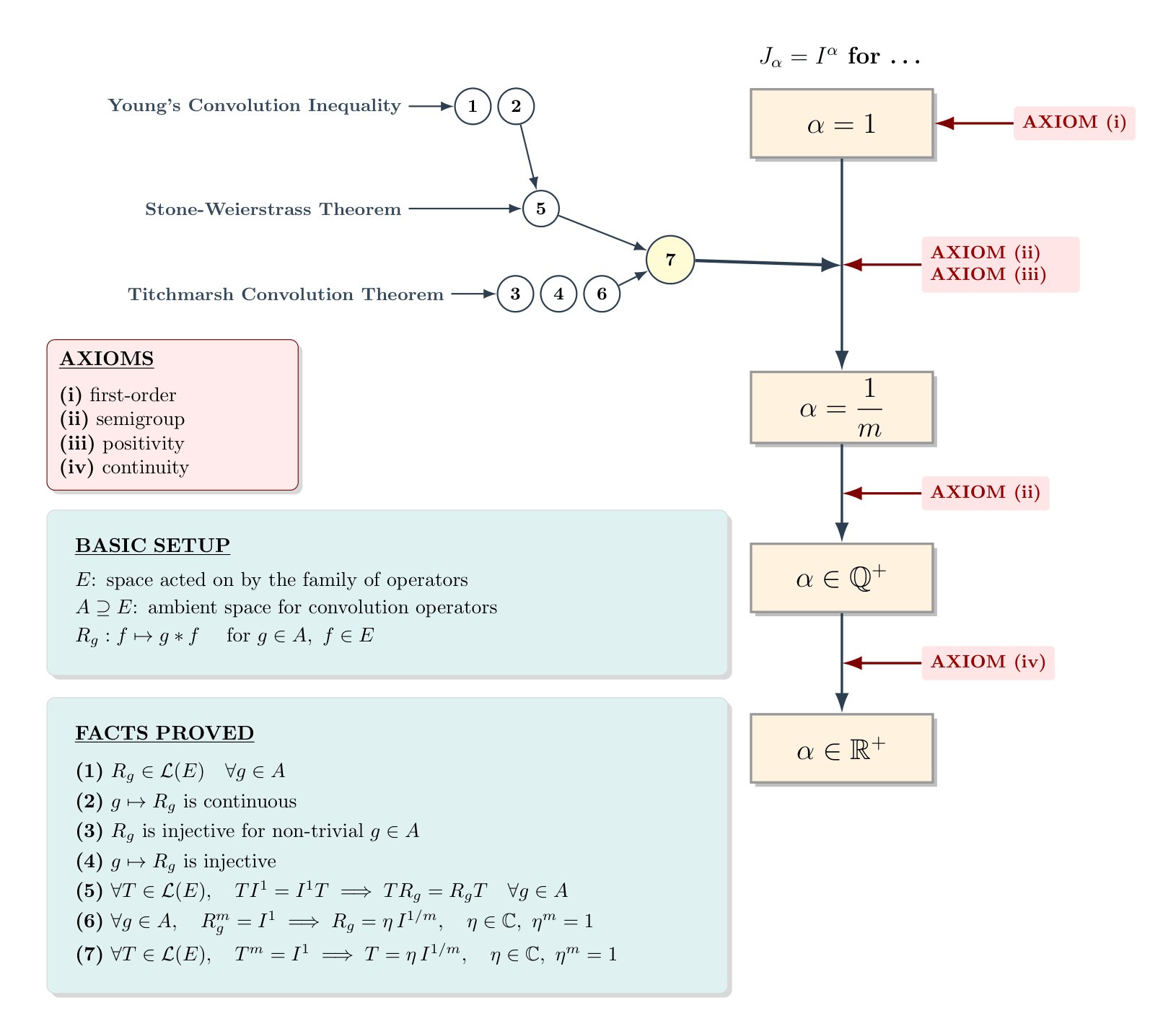}
\caption{The structure of the proof of the Cartwright--McMullen theorem.}
\label{Fig1}
\end{figure}

\section{Improvements to the original result on $[0,1]$} \label{Sec:nopos}

In this section, we present some improvements to the original result of Cartwright and McMullen (Theorem \ref{Thm:CMorig} above) that occurred to us while examining the proof structure.

First of all, the assumption that $1\leq p<\infty$ does not seem to be necessary, as the proof works in just the same way when $p=\infty$. The only potential problem is that the space of polynomials is not dense in $L^{\infty}[0,1]$, but it does not need to be, since denseness of polynomials is only used in the ambient space $A$ (which is $L^1[0,1]$ in every case) rather than the space $E$ which depends on $p$. Thus, we can extend the statement of Theorem \ref{Thm:CMorig} to include $1\leq p\leq\infty$.

Secondly, let us examine axiom (iii), the positivity condition. It is interesting to note that this axiom is only used once in the proof of Theorem \ref{Thm:CMorig}, and only to eliminate the possibility of complex roots of unity appearing as multipliers on the rational-order operators. Intuitively, it feels that this condition is unnecessarily strong: axioms (i) and (ii) are obviously necessary to fix the type of operators we are dealing with, and axiom (iv) enables us to extend the result from $\mathbb{Q}$ to $\mathbb{R}$ in a natural way, but is it really necessary to assume positivity of all the operators just to avoid the complex roots of unity as factors?

It is trivial to see that axiom (iii) can be replaced by the following statement without affecting the result: \emph{
\begin{enumerate}
\item[(iii*)] Weak version of positivity: for every $\alpha>0$, there exists $f\in E$ such that $f>0$ a.e. and $J_\alpha f>0$ a.e.
\end{enumerate}}
\noindent This is enough because, if we already know that $J_{1/m}=\eta\cdot I_0^{1/m}$ for some complex root of unity $\eta$, then having even a single function $f$ with $J_{1/m}f>0$ a.e. and $I_0^{1/m}>0$ a.e. is enough to say that $\eta$ must be a positive real number, therefore must be $1$.

However, we can do even better than this. The following theorem shows that the positivity axiom can be eliminated without introducing any significantly different families of operators. Positivity is necessary for strict uniqueness, as without it there are infinitely many possible families of operators, but all of these are essentially trivial modifications (constant multiples) of the Riemann--Liouville family. The following result is new and appears here for the first time.

\begin{thm}[Cartwright--McMullen theorem without positivity condition] \label{Thm:CMnopos}
Let $E$ be a space of functions on the interval $[0,1]$, either $E=C[0,1]$ or $E=L^p[0,1]$ with $1\leq p<\infty$. Let $\{J_{\alpha}\}_{\alpha>0}$ be a family of bounded linear operators on $E$ satisfying the following three axioms:
\begin{enumerate}[(i)]
\item $\displaystyle \big(J_1f\big)(x)=\int_0^xf(t)\,\mathrm{d}t$ for every $f\in E$ and a.e. $x\in[0,1]$.
\item Semigroup property: $J_{\alpha}J_{\beta}=J_{\alpha+\beta}$ for all $\alpha,\beta>0$.
\item Continuity: the map $\alpha\mapsto J_{\alpha}$ is continuous from $(0,\infty)$ into the space $\mathcal{L}(E)$ under some metric weaker than the operator norm metric on this space.
\end{enumerate}
Then, there exists a fixed $n\in\mathbb{Z}$ such that
\[
J_{\alpha}f=e^{2n\pi i\alpha}\cdot I^{\alpha}_{0}f,
\]
meaning that every operator $J_{\alpha}$ is a constant multiple of the corresponding Riemann--Liouville fractional integral operator with constant of integration $0$, in such a way that the constant multipliers can be interpreted as ``fractional powers of $1$''. Taking $n=0$ leads to the Riemann--Liouville family itself.
\end{thm}

\begin{proof}
Since the positivity condition is only used towards the end of the proof of Theorem \ref{Thm:CMorig}, we can do most of the proof in exactly the same way as Cartwright and McMullen did. In particular, facts 1 to 7 are still true without the positivity condition, so we know the following:
\[
\forall\;m\in\mathbb{N},\:\exists\;\eta\in\mathbb{C}\::\:\eta^m=1\text{ and }J_{1/m}=\eta\cdot I_{0}^{1/m}.
\]
Let $\ell:\mathbb{Q}^+\to S^1\subset\mathbb{C}$ be the map that sends $1/m$ to $\eta$ as given above, and by the semigroup property sends $k/m$ to $\eta^k$. This defines a group homomorphism
\[
\ell:\left(\mathbb{Q},+\right)\longrightarrow\left(S^1,\times\right),
\]
where we define $\ell(0)=0$ and $\ell(-x)=-\ell(x)$ for every $x\in\mathbb{Q}^+$ to get a well-defined homomorphism on the whole of $\mathbb{Q}$.

Recall that we have a mapping
\begin{align*}
J:(0,\infty)&\longrightarrow\mathcal{L}(E) \\
\alpha&\longmapsto J_{\alpha}
\end{align*}
which is a continuous mapping between metric spaces, by the last of the stated axioms on the family $(J_\alpha)$. For any compact interval $[a,b]\subset\mathbb{R}^+$, the mapping $J\big|_{[a,b]}$ is a continuous function from a compact metric space to another metric space, which must be uniformly continuous \cite[Proposition 13.24]{sutherland}. Restricting to the domain $\mathbb{Q}\cap[a,b]$, and using our knowledge of the nature of $J_{\alpha}$ for rational $\alpha$, we therefore have a uniformly continuous map
\begin{align*}
J:\mathbb{Q}\cap[a,b]&\longrightarrow\mathcal{L}(E) \\
\alpha&\longmapsto\ell(\alpha)\cdot I_{0}^{\alpha}.
\end{align*}
Similarly, we have another uniformly continuous map
\begin{align*}
I:\mathbb{Q}\cap[a,b]&\longrightarrow\mathcal{L}(E) \\
\alpha&\longmapsto I_{0}^{\alpha}.
\end{align*}
Thus, the function $\ell$ is also uniformly continuous on the domain $\mathbb{Q}\cap[a,b]$. Every uniformly continuous function on a dense subspace of a metric space can be uniquely extended to a uniformly continuous function on the whole metric space \cite[\S II.3.6]{bourbaki}, so there is a unique extension of $\ell$ to a uniformly continuous function $\ell:[a,b]\to S^1$. Since $[a,b]$ was arbitrary as a compact interval in $\mathbb{R}^+$, this means we have a continuous function
\[
\ell:\left(\mathbb{R},+\right)\longrightarrow\left(S^1,\times\right),
\]
which is still a group homomorphism, by a limiting process starting from the dense subset $\mathbb{Q}$. Finally, it is known \cite{mse} that all continuous homomorphisms from the real numbers under addition to the unit circle under multiplication must be of the form
\[
\ell(\alpha)=e^{ip\alpha},\qquad\alpha\in\mathbb{R},
\]
for some fixed real number $p$.

It remains to determine the nature of $p$. Since $\ell$ maps each rational number $1/m$ to an $m$th root of unity, say $\eta=e^{2\pi ik/m}$ for some $k\in\mathbb{Z}$, the fixed number $p$ must be in $2\pi\mathbb{Z}$, i.e. there exists $n\in\mathbb{Z}$ such that $\ell(\alpha)=e^{2n\pi i\alpha}$ for all $\alpha\in\mathbb{R}$. Since we already showed that $J_{\alpha}=\ell(\alpha)\cdot I_0^\alpha$ for all $\alpha\in\mathbb{Q}^+$, the same is true for all $\alpha\in\mathbb{R}^+$ by the continuity axiom. Thus, the claimed result is proved.
%Since the structure of complex roots of unity is well known, this means the following:
%\[
%\forall\;m\in\mathbb{N},\:\exists\;k\in\mathbb{N}\::\:J_{1/m}=e^{2\pi i k/m}\cdot I_{0}^{1/m}.
%\]
%By axiom (ii), the semigroup property, ee immediately have:
%\[
%\forall\;\tfrac{p}{q}\in\mathbb{Q}^+,\:\exists\;\tfrac{r}{q}\in\mathbb{Q}^+\::\:J_{p/q}=e^{2\pi ir/q}\cdot I_{0}^{p/q}.
%\]
%Let $\ell:\mathbb{Q}^+\to\mathbb{Q}^+$ be the map sending $\tfrac{p}{q}$ to $\tfrac{r}{q}$ for each $\tfrac{p}{q}\in\mathbb{Q}^+$ (since $r$ is not uniquely defined by the definition, we assume ***)
\end{proof}

\begin{remark}
The continuity axiom in Theorem \ref{Thm:CMnopos} is slightly strengthened compared with the continuity axiom in the original Theorem \ref{Thm:CMorig}: it is now assumed that the continuity is into the space $\mathcal{L}(E)$ with a topology that is metrisable, not only Hausdorff. This is needed for our proof of Theorem \ref{Thm:CMnopos}, because we used a fact about continuous mappings between metric spaces. However, we believe this is only a mild strengthening of the assumption, and very reasonable since the standard topology on $\mathcal{L}(E)$ (the one that would be assumed in most applications of our result) is not only metrisable but normable.

On the other hand, the results in the next section will assume a space $E$ which is a Fre\'echet space, not necessarily a Banach space, and therefore the standard topology on $\mathcal{L}(E)$ is not necessarily metrisable. Thus, Theorem \ref{Thm:CMnopos} will not easily extend to the later scenarios, and we will need to make do with a modified (weaker) version of the positivity assumption, along the lines of (iii*) above, rather than eliminating it entirely.
\end{remark}

\section{The result in Fr\'echet spaces of functions on $\mathbb{R}$} \label{Sec:Lploc}

\begin{thm} \label{Thm:Lploc}
Let $E$ be a space of functions on the real line with left-bounded support, either
\[
E = \Big\{ f \in L^p_{\mathrm{loc}}(\mathbb{R}) : \exists\, a \in \mathbb{R} \text{ such that } \operatorname{supp}(f) \subset [a,\infty) \Big\}
\]
for $1\leq p<\infty$, or
\[
E = \Big\{ f \in C^k(\mathbb{R}) : \exists\, a \in \mathbb{R} \text{ such that } \operatorname{supp}(f) \subset [a,\infty) \Big\}
\]
for $k\in\mathbb{Z}^+_0$. There exists a unique family $\{J_{\alpha}\}_{\alpha>0}$ of continuous linear operators on $E$ satisfying the following four axioms:
\begin{enumerate}[(i)]
\item $\displaystyle \big(J_1f\big)(x)=\int_{-\infty}^xf(t)\,\mathrm{d}t$ for every $f\in E$ and a.e. $x\in\mathbb{R}$.
\item Semigroup property: $J_{\alpha}J_{\beta}=J_{\alpha+\beta}$ for all $\alpha,\beta>0$.
\item Positivity: for every $\alpha>0$, if $f\in E$ and $f\geq0$, then $J_{\alpha}f\geq0$ a.e.
\item Continuity: the map $\alpha\mapsto J_{\alpha}$ is continuous from $(0,\infty)$ into the space $\mathcal{L}(E)$ under some Hausdorff topology weaker than the usual seminormed topology on this space.
\item Shift commutation: for every $\alpha > 0$ and every $\tau\in\mathbb{R}$,
    \[
    J^{\alpha} S_\tau = S_\tau J^{\alpha},
    \]
where the shift operators $S_\tau$ are defined by $\left(S_\tau f\right)(x) := f(x-\tau)$.
\end{enumerate}
The unique family satisfying these five axioms is the family of Riemann--Liouville fractional integrals with constant of integration $-\infty$, namely $J_{\alpha}f=I^{\alpha}_{-\infty}f=f*h_{\alpha}$ where the convolution is of Fourier type.
\end{thm}

The proof of Theorem \ref{Thm:Lploc}, step by step, will take up the entirety of this section.

\begin{remark}
The fifth axiom for the family of operators in Theorem \ref{Thm:Lploc} is new and reflects a fundamental structural difference between the bounded-interval setting and the present framework.

One main difference in the proof, compared with the original Cartwright--McMullen approach, lies in the choice of a dense subset of the underlying function space. In the classical $L^1[0,1]$ setting, polynomials can be used as a dense subset, by the Stone--Weierstrass theorem. Although polynomials are still contained in $L^1_{\mathrm{loc}}(\mathbb{R})$, they are no longer dense in this space, because polynomials always grow at infinity and therefore cannot approximate bounded functions which are also locally integrable. One of the key new ideas in the current work lay in finding a new space of functions that is dense in $L^1_{\mathrm{loc}}(\mathbb{R})$ but simple enough to have the desired properties like the space of polynomials in the original Cartwright--McMullen proof. We chose to use the space of step functions, and with this choice, the uniqueness proof required the shift-commutation assumption that is stated in Theorem \ref{Thm:Lploc}.

It is not clear whether the shift-commutation assumption is genuinely necessary for uniqueness itself; however, it is essential for the proof strategy adopted here.
\end{remark}

\begin{remark}
The positivity axiom (iii) can, as discussed in section \ref{Sec:nopos}, be replaced by the following alternative version: \emph{
\begin{enumerate}
\item[(iii*)] Weak version of positivity: for every $\alpha>0$, there exists $f\in E$ such that $f>0$ a.e. and $J_\alpha f>0$ a.e.
\end{enumerate}}
\noindent Again, it will be clear from the structure of the proof below that this is sufficient.
\end{remark}

\subsection{Spaces and convolution structure}

Recall that the proof of Theorem \ref{Thm:CMorig} starts by choosing an ambient space $A$, a smaller function space $E$, and a family of convolution operators. Thus, this subsection is devoted to the setup of our function spaces and convolution, and all necessary properties which we will need.

The space $E$ is already stated in Theorem \ref{Thm:Lploc}. All convolution kernels $g\in A$ will be understood as distributions with left-bounded support, i.e. in the space
\[
\mathcal{D}'_{+}(\mathbb{R}) := \big\{ T \in \mathcal{D}'(\mathbb{R})\: : \:\exists\, a\in\mathbb{R}\ \text{with } \operatorname{supp}(T)\subset [a,\infty) \big\}.
\]
Specifically, we will take
\[
A=\mathcal{D}'_{+}(\mathbb{R})\cap L^1_{\mathrm{loc}}(\mathbb{R})= \Big\{ g \in L^1_{\mathrm{loc}}(\mathbb{R}) : \exists\, a \in \mathbb{R} \text{ such that } \operatorname{supp}(g) \subset [a,\infty) \Big\},
\]
and note that $E\subset A\subset\mathcal{D}'_{+}(\mathbb{R})$.

The convolution operator is defined as
\[
R_g : f \mapsto g * f,
\]
where $f\in E$ and $g\in A$, and where the convolution is taken to be of Fourier type. Note that, if $f\in E$ and $g\in A$, then both $f$ and $g$ have left-bounded support, say
\[
\operatorname{supp}(f) \subset [a,\infty), \qquad \operatorname{supp}(g) \subset [b,\infty)
\]
for some $a,b \in \mathbb{R}$. These assumptions affect the domain of integration for the Fourier-type convolution, and the improper integral from \eqref{convol:F} becomes an integral over a finite domain in this case:
\begin{equation}
\label{eq:fourier_convolution_reduced}
\left(f * g\right)(x) =
\begin{cases}
0, &\text{ if } x < a+b,
\\
\displaystyle \int_{b}^{x-a} f(x-t)g(t)\,\mathrm{d}t, &\text{ if } x \ge a+b.
\end{cases}
\end{equation}
Thus, we have
\[
\operatorname{supp}\left(f * g\right) \subset [a+b,\infty)=\operatorname{supp}(f)+\operatorname{supp}(g).
\]

It is straightforward to prove that Fourier-type convolution is commutative, associative, and bilinear wherever it is defined, either in a distributional sense on $\mathcal{D}'(\mathbb{R})$ or in the usual sense of integration on appropriate function spaces such as $E$ and $A$. Associativity of this convolution also shows that $R_{g_1} \circ R_{g_2} = R_{g_1 * g_2}$, again whenever these are defined.

Convolution of distributions is in general not defined, but in the space $\mathcal{D}'_{+}(\mathbb{R})$, the left-bounded support structure means that convolutions are well-defined. Indeed, we have the following result due to Schwartz.

\begin{propn}[{\cite[Chapter VI, Theorem XIV]{Schwartz}}] \label{Prop:intdom:Schwartz}
Fourier-type convolution turns $\mathcal{D}'_{+}(\mathbb{R})$ into a commutative ring, with the Dirac delta serving as multiplicative identity element, and this ring has no nonzero zero divisors: if $T,S \in \mathcal{D}'_{+}(\mathbb{R})$ and $T*S = 0$, then $T=0$ or $S=0$. This is precisely the statement that $\mathcal{D}'_{+}(\mathbb{R})$ is an integral domain under convolution.
\end{propn}

We can use the above result to obtain the injectivity results corresponding to facts 3 and 4 from the proof of Theorem \ref{Thm:CMorig}. In the original result on $[0,1]$, injectivity was obtained via the Titchmarsh convolution theorem, and the final paragraph of the Cartwright--McMullen paper \cite{cartwright-mcmullen} acknowledges that the integral domain property of distributions will be needed instead in the context of function spaces on $\mathbb{R}$. However, before stating the new injectivity results explicitly, we need the following result showing that $R_g$ does indeed define a mapping in $\mathcal{L}(E)$ for any $g\in A$.

\begin{propn}[Continuity properties] \label{Prop:continuity}
The mapping
\[
E \times A \to E, \qquad (f,g) \mapsto f * g,
\]
is bilinear and separately continuous. This proves the analogue of facts 1 and 2 from the proof of Theorem \ref{Thm:CMorig}: namely, that $R_g : E \to E$ is a continuous linear operator for every $g\in A$ and the mapping $g\mapsto R_g$ is continuous from $A$ into $\mathcal{L}(E)$.
\end{propn}

\begin{proof}
Let $f\in E$ with $\operatorname{supp}(f) \subset [a,\infty)\}$, and $g \in A$ with $\operatorname{supp}(g) \subset [b,\infty)$, and let $K \subset \mathbb{R}$ be compact. We will bound the $K$-seminorm of $f*g$ (a norm considered just on the domain $K$) in terms of appropriate seminorms of $f$ and $g$ in their respective spaces $E$ and $A$. Bilinearity is clear.

Since $\operatorname{supp}(f*g) \subset [a+b,\infty)$, we may assume
\[ K \subset [a+b, M] \]
for some $M > a+b$.

Recall that there are two possible cases for $E$ in the statement of Theorem \ref{Thm:Lploc}: it can be the subset with left-bounded support of either $L^p_{\mathrm{loc}}(\mathbb{R})$ or $C^k(\mathbb{R})$. We must deal with these two cases separately.

Firstly, we consider the $L^p_{\mathrm{loc}}(\mathbb{R})$ case and assume $1\leq p<\infty$. For $x \in K$, we have
\[ (f * g)(x) = \int_{b}^{\,x-a} f(x-t)\, g(t)\, \mathrm{d}t \le \int_{b}^{\,x-a} |f(x-t)|\, |g(t)|\, \mathrm{d}t. \]
Integrating over $x \in K$ and applying Minkowski's integral inequality \cite[Theorem 6.19]{folland} yields
\begin{align*}
\|f * g\|_{L^p(K)}
&= \left( \int_K \left| \int_{b}^{\,x-a} f(x-t)\, g(t)\, \mathrm{d}t \right|^p \,\mathrm{d}x \right)^{1/p} \\
&\le \int_{b}^{\,M-a} |g(t)| \left( \int_K |f(x-t)|^p \,\mathrm{d}x \right)^{1/p} \,\mathrm{d}t.
\end{align*}
For each fixed $t \in [b, M-a]$, the change of variables $u = x - t$ gives
\[ \left( \int_K |f(x-t)|^p \,\mathrm{d}x \right)^{1/p}
= \left( \int_{K-t} |f(u)|^p \,\mathrm{d}u \right)^{1/p}. \]
Since $\operatorname{supp}(f) \subset [a,\infty)$, we may restrict the domain of integration:
\[ \int_{K-t} |f(u)|^p \,\mathrm{d}u
=
\int_{(K-t)\cap[a,\infty)} |f(u)|^p \,\mathrm{d}u. \]
Now $K \subset [a+b, M]$ implies $K - t \subset [a+b-t,\, M-t]$, and since $t \ge b$ we have $M - t \le M - b$. Therefore,
\[ (K - t) \cap [a,\infty) \subset [a,\, M-b]. \]
Define the compact sets
\[ K' := [a,\,M-b]\qquad\text{ and }\qquad K'' := [b,\,M-a]. \]
Then, for all $t \in K''$, we have
\[ \left( \int_K |f(x-t)|^p \,\mathrm{d}x \right)^{1/p}
=
\|f\|_{L^p((K-t)\cap[a,\infty))}
\le \|f\|_{L^p(K')} = q_{K'}(f). \]
Consequently,
\begin{equation}
\label{eq:local_boundedness_kernel}
q_K(f * g) =\|f * g\|_{L^p(K)} \le q_{K'}(f)\cdot \int_{b}^{\,M-a} |g(t)|\,\mathrm{d}t = q_{K'}(f)\cdot q_{K''}(g),
\end{equation}
where each $q$ is defined according to the appropriate function space, namely $q_{K}(f*g)$ and $q_{K'}(f)$ as the $L^p$-seminorms for $f,f*g\in E$ and $q_{K''}(g)$ as the $L^1$-seminorm for $g\in A$.

Thus, we have separate continuity as stated, i.e. both continuity of each $R_g$ and continuity of the mapping $g\mapsto R_g$ with respect to the locally convex topology generated by the seminorm family $\{q_K\}$. This completes the proof in the case that $E$ is the left-bounded support subset of $L^p_{\mathrm{loc}}(\mathbb{R})$ with $1\leq p<\infty$.

Secondly, we consider the case of $L^p_{\mathrm{loc}}(\mathbb{R})$ with $p=\infty$. Then, the inequality involving seminorms is even easier to obtain:
\[
(f * g)(x) \le \int_{b}^{\,x-a} |f(x-t)|\, |g(t)|\, \mathrm{d}t \le \esssup_{[a,x-b]}\big|f\big|\cdot\int_{b}^{\,x-a} |g(t)|\, \mathrm{d}t,
\]
therefore
\[
q_K(f*g)=\esssup_{K}\big|f*g\big| \le \esssup_{[a,M-b]}\big|f\big|\cdot\int_{b}^{\,M-a} |g(t)|\, \mathrm{d}t = q_{K'}(f)\cdot q_{K''}(g),
\]
and the remainder of the proof goes exactly as before.

% $K' = [a, M-b]$ depends on $K$, on the support bound $b$ of $g$, and on the fixed support threshold $a$ of the subspace $E_a$ under consideration.
%
%\medskip
%\noindent\textbf{Continuity.}
%The estimate \eqref{eq:local_boundedness_kernel} shows that for each fixed $g \in A$ and each compact $K \subset \mathbb{R}$, there exist a compact $K' \subset \mathbb{R}$ and a constant $C_{K,g} > 0$ such that
%
%\[ q_K(R_g f) \le C_{K,g}\, q_{K'}(f) \qquad \text{for all } f \in E. \]

Thirdly, we consider the case of $C^k(\mathbb{R})$. Note that
\[
(f*g)'(x)=\frac{\mathrm{d}}{\mathrm{d}x}\int_{-\infty}^{\infty}f(x-t)g(t)\,\mathrm{d}t=\int_{-\infty}^{\infty}f'(x-t)g(t)\,\mathrm{d}t=(f'*g)(x),
\]
and similarly
\[
(f*g)^{(n)}=f^{(n)}*g
\]
for any $n\in\mathbb{N}$ such that $f$ is $n$ times differentiable. Thus,
\[
q_K(f*g) = \sup_{K}\big|f*g\big|+\sup_{K}\big|f'*g\big|+\sup_{K}\big|f''*g\big|+\cdots+\sup_{K}\big|f^{(k)}*g\big|,
\]
and then the remainder of the proof is the same as in the case of $L^p_{\mathrm{loc}}(\mathbb{R})$ with $p=\infty$.
\end{proof}

Next, in order to use Proposition \ref{Prop:intdom:Schwartz} in our function space setting, we should recall how functions are interpreted as distributions. Every $f \in L^p_{\mathrm{loc}}(\mathbb{R})\subset L^1_{\mathrm{loc}}(\mathbb{R})$ defines a distribution $T_f \in \mathcal{D}'(\mathbb{R})$ by
\begin{equation}
\label{eq:dist_associated_to_f}
T_f(\varphi) := \int_{\mathbb{R}} f(x)\,\varphi(x)\,dx,
\qquad \varphi \in \mathcal{D}(\mathbb{R}).
\end{equation}
Moreover, if $\operatorname{supp}(f)\subset [a,\infty)$, then $T_f$ vanishes for every test function $\varphi$ supported in $(-\infty,a]$, and therefore
\[
\operatorname{supp}(T_f)\subset [a,\infty).
\]
In the distributional sense, we have
\begin{equation} \label{eq:supp_Tf_in_supp_f}
\operatorname{supp}(T_f)=\operatorname{supp}(f).
\end{equation}
The convolutions are also compatible: if $f,g \in L^1_{\mathrm{loc}}(\mathbb{R})$ are such that their Fourier-type convolution $f*g$ is well-defined almost everywhere as a function, then the associated distributions satisfy
\begin{equation}
\label{eq:dist_conv_compatibility}
T_{f*g} \;=\; T_f * T_g \qquad \text{in } \mathcal{D}'(\mathbb{R}).
\end{equation}
In particular, if the function $f*g = 0$ almost everywhere, then $T_f * T_g = 0$ as a distribution.

\begin{propn}[Injectivity properties] \label{Prop:E_integral_domain}
Let $f \in E$, $g\in A$, and assume that $f*g = 0$ a.e. on $\mathbb{R}$. Then either $f = 0$ a.e. on $\mathbb{R}$ or $g = 0$ a.e. on $\mathbb{R}$. This proves the analogue of facts 3 and 4 from the proof of Theorem \ref{Thm:CMorig}: namely, that $R_g : E \to E$ is injective for every $g\neq0$ and the mapping $g\mapsto R_g$ is injective from $A$ into $\mathcal{L}(E)$.
\end{propn}

\begin{proof}
Since $f,g \in E \subset L^p_{\mathrm{loc}}(\mathbb{R}) \subset L^1_{\mathrm{loc}}(\mathbb{R})$, they define
distributions $T_f, T_g \in \mathcal{D}'(\mathbb{R})$ by \eqref{eq:dist_associated_to_f}.
The support assumptions on $f$ and $g$ imply that $T_f, T_g \in \mathcal{D}'_{+}(\mathbb{R})$ by \eqref{eq:supp_Tf_in_supp_f}.
If $f*g = 0$ almost everywhere, then $T_{f*g}=0$.
By the compatibility \eqref{eq:dist_conv_compatibility}, we have

\[ T_f * T_g = T_{f*g} = 0 \qquad \text{in } \mathcal{D}'(\mathbb{R}). \]
Since $\mathcal{D}'_{+}(\mathbb{R})$ is an integral domain under convolution, it follows that either $T_f = 0$ or $T_g = 0$. But $T_f=0$ means $\int f\varphi = 0$ for all $\varphi \in \mathcal{D}(\mathbb{R})$, hence $f=0$ almost everywhere;
similarly $T_g=0$ implies $g=0$ almost everywhere.
\end{proof}

\subsection{Commutation with Convolution Operators}
\label{subsec:commutation}

Up to now, we have established the analogues of facts 1 to 4 from the proof of Theorem \ref{Thm:CMorig}. In the original Cartwright--McMullen work, proving fact 5 relied on the Stone--Weierstrass theorem and the denseness of polynomials in $L^1[0,1]$. Now, the ambient space $A$ has changed from $L^1[0,1]$ to $L^1_{\mathrm{loc}}(\mathbb{R})$, in which polynomials are no longer dense (for example, any function that is bounded everywhere and non-constant cannot be approximated on the whole of $\mathbb{R}$ by polynomials). Thus, we need a different argument instead of using the Stone--Weierstrass theorem. The best way we found to get around this issue was by using step functions, which are dense in $L^1_{\mathrm{loc}}(\mathbb{R})$, but then for the commutativity result we need an extra assumption. This is why the fifth axiom appeared in our Theorem \ref{Thm:Lploc}; we do not know if it is mathematically necessary for the truth of the result, but it is necessary for the method of proof that we have constructed.

First we need the following known result.

\begin{lem}[{\cite[Chapter 7, Proposition 10]{royden}}] \label{thm:step_approx}
Let \([a,b]\) be a closed, bounded interval and \(1 \le p < \infty\). Then the subspace of step functions on \([a,b]\) is dense in \(L^p{[a,b]}\). 
\end{lem}

We adapt this density result on a compact interval to a density result on the whole of $\mathbb{R}$ as follows.

\begin{thm} \label{Thm:stepdense}
Step functions with left-bounded support are dense in the following space:
\[
A := \big\{ f \in L^1_{\mathrm{loc}}(\mathbb{R}) : \operatorname{supp}(f)\subset [a,\infty)
\text{ for some } a\in\mathbb{R} \big\},
\]
with respect to the locally convex topology induced by the seminorms
\[
q_K(f):=\|f\|_{L^1(K)}, \qquad K\subset\mathbb{R} \text{ compact}.
\]
\end{thm}

\begin{proof}

Let $f \in L^1_{\mathrm{loc}}(\mathbb{R})$ with $\operatorname{supp}(f) \subseteq [a,\infty)$, and fix $\varepsilon > 0$. We aim to approximate $f$ by a step function $s \in L^1_{\mathrm{loc}}(\mathbb{R})$ with the same support such that
\[ \|f - s\|_{L^1(K)} < \varepsilon \]
for all compact sets $K \subset \mathbb{R}$.

\medskip

\textbf{Step 1.} Consider the half-line $[a,\infty)$, and cover it by a sequence of compact intervals, for example:
\[ [a,a+1], \ [a+1,a+10], \ [a+10,a+100], \ \ldots \]
Denote these intervals by $K_n = [a + 10^{n-2}, \, a + 10^{n-1}]$ for $n\geq2$, $K_1=[a,a+1]$.

\medskip

\textbf{Step 2.} For the first interval $K_1 = [a,a+1]$, find a step function $s_1$ supported in $K_1$ such that
\[ \|f - s_1\|_{L^p(K_1)} < \frac{\varepsilon}{2}. \]
For the second interval $K_2 = [a+1,a+10]$, approximate $f$ by a step function $s_2$ supported in $K_2$ such that
\[ \|f - s_2\|_{L^p(K_2)} < \frac{\varepsilon}{4}. \]
In general, for each interval $K_n$, we find a step function $s_n$ supported in $K_n$ such that
\[ \|f - s_n\|_{L^p(K_n)} < \frac{\varepsilon}{2^n}. \]
This is possible for every $n$ by Lemma \ref{thm:step_approx}, since every $K_n$ is compact. Each step function $s_n$ is supported on a compact interval and is $L^p$-integrable on that interval, hence $s_n \in L^p_{\mathrm{loc}}(\mathbb{R})$.

\medskip

\textbf{Step 3.} Define the function
\[ s := \sum_{n=1}^{\infty} s_n . \]
This is well-defined and supported in $[a,\infty)$, because for each $x \ge a$, there exists a unique index $n(x)$ such that $x \in K_{n(x)}$, and therefore $s_m(x)=0$ for all $m \neq n(x)$, giving $s(x) = s_{n(x)}(x)$, so the series is actually a finite sum at each point. In particular, $s \in L^p_{\mathrm{loc}}(\mathbb{R})$ and $\operatorname{supp}(s) \subseteq [a,\infty)$.

Let $K \subset \mathbb{R}$ be compact. Then $K \cap [a,\infty)$ is bounded, so there exists $N\in\mathbb{N}$ such that
\[ K \cap [a,\infty) \subset \bigcup_{n=1}^N K_n. \]
Then, $s_n=0$ almost everywhere on $K$ for all $n>N$, and hence $s = \sum_{n=1}^N s_n$ almost everywhere on $K$. Therefore,
\[ \begin{aligned}
\|f - s\|_{L^p(K)}
&= \left\|\, f - \sum_{n=1}^N s_n \,\right\|_{L^p(K)} \le \sum_{n=1}^N \|f - s_n\|_{L^p(K\cap K_n)} \\ &\le \sum_{n=1}^N \|f - s_n\|_{L^p(K_n)} < \sum_{n=1}^N \frac{\varepsilon}{2^n} < \varepsilon. \end{aligned} \]
Since $s$ is a step function with left-bounded support and $K\subset\mathbb{R}$ is an arbitrary compact subset, the proof is complete.
\end{proof}

Armed with the above result, we are able to establish the following.

\begin{thm} \label{Thm:commutingRg}
If $T\in\mathcal{L}(E)$ commutes with all shift operators (satisfies axiom (v) in the statement of Theorem \ref{Thm:Lploc}), then $TR_{h_1}=R_{h_1}T$ implies $TR_g=R_gT$ for every $g\in A$, where $E$ and $A$ are as defined earlier in this section and $h_1$ is the Heaviside function according to \eqref{h:fns}. This is the analogue of fact 5 from the proof of Theorem \ref{Thm:CMorig}.
\end{thm}

\begin{proof}
Let $f\in E$ and consider an arbitrary step function $g$ with left-bounded support as follows:
\[ g(x)=\sum_{k=1}^{\infty} c_k\,\chi_{(\alpha_k,b_k)}(x),
\qquad (\alpha_k, b_k) \subset (0, \infty) . \]
Then, for any $x\in\mathbb{R}$,
\[ (R_g f)(x)=(g*f)(x)=\sum_{k=1}^n c_k \int_{\alpha_k}^{b_k} f(x-t)\,\mathrm{d}t. \]
With the change of variables $u=x-t$, we get
\[ \int_{\alpha_k}^{b_k} f(x-t)\,\mathrm{d}t
=\int_{x-b_k}^{x-\alpha_k} f(u)\,\mathrm{d}u
=\int_{-\infty}^{x-\alpha_k} f(u)\,\mathrm{d}u-\int_{-\infty}^{x-b_k} f(u)\,\mathrm{d}u. \]
Recalling that
\[ (R_{h_1} f)(x)=\int_{0}^{\infty} f(x-t)\,\mathrm{d}t=\int_{-\infty}^{x} f(u)\,\mathrm{d}u, \]
and therefore
\[
(S_cR_{h_1} f)(x)=\int_{-\infty}^{x-c} f(u)\,\mathrm{d}u,\qquad c\in\mathbb{R},
\]
for any $c\in\mathbb{R}$ where $S_c$ is a shift operator as in Theorem \ref{Thm:Lploc}, we may rewrite the above as
\[ \int_{\alpha_k}^{b_k} f(x-t)\,\mathrm{d}t
=(R_{h_1} f)(x-\alpha_k)-(R_{h_1} f)(x-b_k)=(S_{\alpha_k} R_{h_1} T f)(x) - (S_{b_k} R_{h_1} T f)(x). \]
Hence,
\[ R_g f = \sum_{k=1}^{\infty} c_k \big[ S_{\alpha_k} R_{h_1} f - S_{b_k} R_{h_1} f \big].\]
Applying $T$, we have:
\begin{align*}
R_gTf &= \sum_{k=1}^{\infty} c_k \big[ S_{\alpha_k} R_{h_1} T f - S_{b_k} R_{h_1} T f \big]
= \sum_{k=1}^{\infty} c_k \big[ S_{\alpha_k} T R_{h_1} f - S_{b_k} T R_{h_1} f \big],
\\
T R_g f &= T \Bigg[ \sum_{k=1}^{n} c_k \big( S_{\alpha_k} R_{h_1} f - S_{b_k} R_{h_1} f \big) \Bigg]
= \sum_{k=1}^{n} c_k \big( T S_{\alpha_k} R_{h_1} f - T S_{b_k} R_{h_1} f \big).
\end{align*}
Under the shift commutation assumption ($T S_\alpha = S_\alpha T$ for all $\alpha\in\mathbb{R}$), these two expressions are identical. Since $f\in E$ was arbitrary, this means $R_gT=TR_g\in\mathcal{L}(E)$ for any step function $g$ as above.

By Theorem \ref{Thm:stepdense}, such functions $g$ are dense in $A$. By Proposition \ref{Prop:continuity}, since the map $g\mapsto R_g$ is continuous, this means that if $T$ commutes with $R_g$ for all step functions $g$, then $T$ commutes with $R_g$ for all $g\in A$, as required.
\end{proof}

\subsection{The $m$th Root with Roots of Unity}
\label{subsec:mth_root}

\begin{lem}[Characterising $m$th roots among convolution operators] \label{lem:root_inside_Rg}
Let $m\in\mathbb N$ and $\alpha\in\mathbb{R}^+$, and suppose that $g\in A$ satisfies
\[ R_g^m = I_{-\infty}^{\alpha}. \]
Then there exists an $m$th root of unity $\eta\in\mathbb{C}$ such that
\[ R_g = \eta\cdot I_{-\infty}^{\alpha/m}, \]
or equivalently, $g = \eta\, h_{\alpha/m}$ in $A$, where we recall the $h$ functions from \eqref{h:fns}. This is (a slight generalisation of) the analogue of fact 6 from the proof of Theorem \ref{Thm:CMorig}.
\end{lem}

\begin{proof}
Let $g^{(*m)}$ denote the $m$-fold convolution of $g$ with itself. Then we have
\[ R_{g^{(*m)}} = R_g^m = I_{-\infty}^{\alpha}=R_{h_\alpha}. \]
By the injectivity property from Proposition \ref{Prop:E_integral_domain}, we obtain
\[ g^{(*m)} = h_\alpha \quad \text{as kernels in }A. \]
Therefore, in the convolution algebra,
\[
0=g^{(*m)}-h_\alpha=g^{(*m)}-h_{\alpha/m}^{(*m)} = \big(g-\eta_1 h_{\alpha/m}\big) * \cdots * \big(g-\eta_m h_{\alpha/m}\big), \]
where $\eta_k = e^{2\pi i k/m}$ ($k=1,\cdots,m$) are the $m$th roots of unity in $\mathbb{C}$. By the integral domain property (Proposition \ref{Prop:intdom:Schwartz}), one of the $m$ factors must vanish. Hence $g=\eta\, h_{\alpha/m}$ for some $m$th root of unity $\eta$, as claimed.
\end{proof}

\begin{remark} \label{Rem:extending mth root}
Taking $\alpha=1$ in Lemma \ref{lem:root_inside_Rg} allows us to characterise the $m$th roots of the first-order integral operator $I^{1}_{-\infty}$ among convolution operators, and it is this result that we seek to extend (in Lemma \ref{lem:root_without_assuming_Rg} below) to the case where convolution operators are not assumed. The original paper of Cartwright and McMullen \cite{cartwright-mcmullen} did not mention a result in the form of Lemma \ref{lem:root_inside_Rg}, only the $\alpha=1$ case, but in fact it is the $\alpha=m+1$ case of Lemma \ref{lem:root_inside_Rg} that will be required in the proof below.
\end{remark}

\begin{lem}[Characterising $m$th roots without assuming convolution operators]
\label{lem:root_without_assuming_Rg}
Let $m\in\mathbb N$ and suppose $T\in\mathcal L(E)$ commutes with all shift operators (satisfies axiom (v) in the statement of Theorem \ref{Thm:Lploc}) and satisfies
\[
T^m = I_{-\infty}^1.
\]
Then there exists an $m$th root of unity $\eta\in\mathbb{C}$ such that
\[ R_g = \eta\cdot I_{-\infty}^{1/m}, \]
or equivalently, $T=R_g$ with $g = \eta\, h_{1/m}$ in $A$, with $h_{\alpha}$ as defined in \eqref{h:fns}. This is the analogue of fact 7 from the proof of Theorem \ref{Thm:CMorig}.
\end{lem}

\begin{proof}
Recall that $I_{-\infty}^1=R_{h_1}$, and let $S := R_{h_1}T$. Note that $T^m=R_{h_1}$ implies $TR_{h_1}=R_{h_1}T$, therefore $S R_{h_1} = R_{h_1} S$. By Theorem \ref{Thm:commutingRg}, this implies that $T R_g = R_g T$ and $S R_g = R_g S$ for all $g\in A$. In particular, let us take $g=T(h_1)$, the $T$-image of the Heaviside kernel in $A$. Then we have, for any $f\in E\subset A$:
\[
S(f) = TR_{h_1}(f) = T(h_1*f) = T(f*h_1) = TR_f(h_1) = R_fT(h_1) = R_fg = f*g = g*f = R_gf,
\]
which means $S=R_g$ is a convolution operator. Also,
\[ S^m = (R_{h_1}T)^m = R_{h_1}^{m}T^m = R_{h_1}^{m}R_{h_1} = R_{h_1}^{m+1} = R_{h_{m+1}}.\]
By Lemma~\ref{lem:root_inside_Rg} in the case $\alpha=m+1$, we obtain
\[ S = \eta\, R_{h_{(m+1)/m}} = \eta\, R_{h_{1 + 1/m}} = \eta\, R_{h_1} R_{h_{1/m}},\]
for some $m$th root of unity $\eta\in\mathbb{C}$. Therefore,
\[ R_{h_1}T = S = \eta\, R_{h_1} R_{g_{1/m}}. \]
Since $R_{h_1}$ is injective, we can cancel $R_{h_1}$ from the left and conclude that $T = \eta\, R_{h_{1/m}}$, as required.
\end{proof}

\subsection{Proof of Theorem \ref{Thm:Lploc}}

The proof of Theorem \ref{Thm:Lploc} can now be completed quickly using the results from the above subsections. Firstly, it is a classical fact that Riemann--Liouville fractional integrals satisfy the five given axioms, so it remains to prove uniqueness.

The case $\alpha=1$ is clear, since $J_1=I_{-\infty}^{1}$ by the given axiom (i).

The case $\alpha=1/m$ ($m\in\mathbb{N}$) follows from Lemma \ref{lem:root_without_assuming_Rg}, using the semigroup axiom (ii) to know that $J_{1/m}$ must be an $m$th root of $I_{-\infty}^{1}$ and the shift-commutation axiom (v) as it is part of the assumptions in Lemma \ref{lem:root_without_assuming_Rg}, and then the positivity axiom (iii) to say that the $m$th root of unity $\eta$ from the statement of Lemma \ref{lem:root_without_assuming_Rg} must in fact be $1$, since both $J_{1/m}$ and $I_{-\infty}^{1/m}$ are positive operators so their multiplying factor $\eta$ must be a positive real number. Here we note that the weaker positivity axiom (iii*) would also be sufficient for our proof to work.

The case $\alpha\in\mathbb{Q}^+$ then follows from the semigroup axiom (ii).

Finally, we use the continuity axiom (iv) to get the full result for $\alpha\in\mathbb{R}^+$. Note that requiring a Hausdorff topology on $\mathcal{L}(E)$ is still a reasonable assumption since the natural topology on $\mathcal{L}(E)$ is Hausdorff whenever $E$ is a Fr\'echet space, by Remark \ref{Rem:LEfrechet}. \qed

%\section{The result in $C^k(\mathbb{R})$}

\section{The result in a space of distributions} \label{Sec:distrib}

\begin{thm} \label{Thm:D'}
Let $E$ be the space of distributions on the real line with left-bounded support:
\[
E = \mathcal{D}'_{+}(\mathbb{R}) := \big\{ T \in \mathcal{D}'(\mathbb{R})\: : \:\exists\, a\in\mathbb{R}\ \text{with } \operatorname{supp}(T)\subset [a,\infty) \big\}.
\]
There exists a unique family $\{J_{\alpha}\}_{\alpha>0}$ of continuous linear operators on $E$ satisfying the following four axioms:
\begin{enumerate}[(i)]
\item $J_1f=f*h_1$, i.e. convolution with the Heaviside function, for every $f\in E$.
\item Semigroup property: $J_{\alpha}J_{\beta}=J_{\alpha+\beta}$ for all $\alpha,\beta>0$.
\item (A weak version of) positivity: for every $\alpha>0$, there exists a positive function $f\in L^1_{\mathrm{loc}}(\mathbb{R})\subset E$ with $f>0$ a.e. and $\langle J_{\alpha}f,\phi\rangle>0$ for some $\phi\in\mathcal{D}(\mathbb{R})$.
\item Continuity: the map $\alpha\mapsto J_{\alpha}$ is continuous from $(0,\infty)$ into the space $\mathcal{L}(E)$ under some Hausdorff topology weaker than the usual seminormed topology on this space.
\item Shift commutation: every $J^{\alpha}$ commutes with every shift operator $S_\tau$, defined by extending $\left(S_\tau f\right)(x) := f(x-\tau)$ naturally to distributions.
\end{enumerate}
The unique family satisfying these five axioms is the family of Riemann--Liouville fractional integrals with constant of integration $-\infty$, namely $J_{\alpha}f=I^{\alpha}_{-\infty}f=f*h_{\alpha}$.
\end{thm}

\begin{proof}
In this case, the spaces $E$ and $A$ are both the same, namely $\mathcal{D}'_{+}(\mathbb{R})$, as it is not reasonable for $A$ to be any bigger than this space $E$. Convolution operators $R_g : f \mapsto g * f$ are defined using Fourier-type convolution. Proposition \ref{Prop:intdom:Schwartz} immediately gives injectivity of every $R_g$ and of the map $g\mapsto R_g$. Convolution is also bilinear and separately continuous \cite[Chapter 3]{beffa}, so already we have the analogues of facts 1, 2, 3, 4 from the proof of Theorem \ref{Thm:CMorig}.

Theorem \ref{Thm:stepdense} shows that step functions with left-bounded support are dense in the set of $L^1_{\mathrm{loc}}(\mathbb{R})$ functions with left-bounded support, which is weakly dense in $\mathcal{D}'_{+}(\mathbb{R})$ by a well-known result \cite[Proposition 9.5]{folland}. Then the analogue of Theorem \ref{Thm:commutingRg} is still true in the space $\mathcal{D}'_{+}(\mathbb{R})$: if $T\in\mathcal{L}(E)$ is a continuous linear operator that commutes with $R_{h_1}$ and with every shift operator, then we can say that $T$ commutes with $R_g$ for every step function $g$. Then, for a general $g\in\mathcal{D}'_{+}(\mathbb{R})$, we can find a sequence $(g_n)$ of step functions such that
\begin{equation} \label{testfn:limit}
\langle g,\phi \rangle = \lim_{n\to\infty}\langle g_n,\phi \rangle\qquad\text{ for all }\phi\in\mathcal{D}(\mathbb{R}).
\end{equation}
For any $f\in\mathcal{D}'_{+}(\mathbb{R})$ and any $n\in\mathbb{N}$, since $g_n$ is a step function,% and therefore expressible as a linear combination of shift operators applied to $h_1$,
we have
\[
T\big(g_n*f\big)=g_n*Tf\in\mathcal{D}'_{+}(\mathbb{R}),
\]
which means that, for any $\phi\in\mathcal{D}(\mathbb{R})$, we have
\[
\langle T\big(g_n*f\big),\phi\rangle=\langle g_n*Tf,\phi\rangle,
\]
and then we can use a duality argument together with \eqref{testfn:limit} to get the same result with $g_n$ replaced by $g$. Thus, $T$ commutes with $R_g$ for every $g\in\mathcal{D}'_{+}(\mathbb{R})$, and we have the analogue of fact 5 from the proof of Theorem \ref{Thm:CMorig}.

Finally, the analogues of facts 6 and 7 from the proof of Theorem \ref{Thm:CMorig} can be proved in exactly the same way as Lemma \ref{lem:root_inside_Rg} and Lemma \ref{lem:root_without_assuming_Rg}, since the proof of Lemma \ref{lem:root_inside_Rg} just uses the integral domain property from the space of distributions and the proof of Lemma \ref{lem:root_without_assuming_Rg} does not really use any properties of the spaces $E$ and $A$ except those established in the earlier numbered facts from the proof of Theorem \ref{Thm:CMorig}.

Now the proof of Theorem \ref{Thm:D'} goes as follows.
\begin{itemize}
\item The case $\alpha=1$ is clear, since $J_1=I_{-\infty}^{1}$ by the given axiom (i).

\item For the case $\alpha=1/m$ and $m\in\mathbb{N}$, axioms (ii) and (v) together with the analogue of Lemma \ref{lem:root_without_assuming_Rg} tell us that $J_{1/m}$ must be a scalar complex $m$th root of unity multiplied by $I_{-\infty}^{1/m}$. Now we need to use axiom (iii): there is a function $f\in L^1_{\mathrm{loc}}(\mathbb{R})\subset E$ and a test function $\phi\in\mathcal{D}(\mathbb{R})$ such that $f>0$ a.e. and $\langle J_{1/m}f,\phi\rangle>0$, but we also know that $I_{-\infty}^{1/m}$ is a positive operator, thus $I_{-\infty}^{1/m}f\geq0$ a.e. and $\langle I_{-\infty}^{1/m}f,\phi\rangle>0$. Then the complex multiplier $\eta$ between $J_{1/m}f$ and $I_{-\infty}^{1/m}f$ must be a positive real number, and the only positive real root of unity is $1$, so we have $J_{1/m}=I_{-\infty}^{1/m}$.

\item The case $\alpha\in\mathbb{Q}^+$ then follows from the semigroup axiom (ii).

\item Finally, we use the continuity axiom (iv) to get the full result for $\alpha\in\mathbb{R}^+$.
\end{itemize}
\end{proof}

\begin{remark}
The positivity axiom here has been significantly weakened, since positivity of distributions is more complicated to define and is not necessary for the proof to work. The same weakening of the positivity axiom could have been used in Theorem \ref{Thm:Lploc}, with the same logical argument in the proof.

Indeed, it may be logically possible to remove the positivity axiom entirely, as we did for the compact interval setting in Theorem \ref{Thm:CMnopos}. However, we were not able to do this either in the case of Theorem \ref{Thm:Lploc} (Fr\'echet spaces of functions on $\mathbb{R}$) or in the case of Theorem \ref{Thm:D'} (spaces of distributions), because in these cases the space $\mathcal{L}(E)$ is not necessarily metrisable and therefore our strengthening of the continuity axiom in Theorem \ref{Thm:CMnopos} to be into a metric space (rather than just a Hausdorff topological space) may not make sense.
\end{remark}

\section{Conclusions} \label{Sec:concl}

In this paper, we have revisited a result from 1978 \cite{cartwright-mcmullen} in order to cast new light on it and extend it to new settings. Our main contributions are as follows.
\begin{itemize}
\item In Figure \ref{Fig1}, we have shown, in an explicit and visual way, the structure of Cartwright and McMullen's proof of their original result (Theorem \ref{Thm:CMorig} in our paper). This has been very helpful for us in understanding how to modify the proof for other contexts.
\item In Theorem \ref{Thm:CMnopos}, we have proved that one of the original four axioms can essentially be removed, without significantly changing the result.
\item In Theorem \ref{Thm:Lploc}, we have extended the Cartwright--McMullen theorem and proof to the setting of Fr\'echet spaces of functions defined on the whole real line with left-bounded support. This required some significant modifications to the method of proof and even an extra axiom to be added in the statement of the result.
\item In Theorem \ref{Thm:D'}, we have extended the theorem and proof even further, to the setting of spaces of distributions with left-bounded support. The proof in this case is largely the same as for Theorem \ref{Thm:Lploc}, with a few modifications needed for the distributional setting.
\end{itemize}

We believe that this work opens up the topic of unique axiomatisation of fractional integral operators to many new potential directions of investigation, some of which we briefly mention as follows.

Here, we have considered two types of convolution operation: Laplace type and Fourier type. In the literature, many other convolutions have been considered \cite{dimovski,yakubovich-luchko}, and it may be interesting to consider what sort of uniqueness results could be proved with respect to other types of convolution. Also, there are many other function spaces that could be considered, in addition to the spaces of continuously differentiable or integrable functions that we have worked with here, and for some functional analysis researchers it might be useful to have uniqueness results in other domains. As other researchers have done recently \cite{caolabora,jornet}, it may also be possible to prove uniqueness results for other families of operators, including multi-parameter ones, in addition to the classical Riemann--Liouville fractional integrals. Generally, proving uniqueness results for known families of fractional integral operators will provide important knowledge about the structure of the whole field of fractional calculus, and we believe that this is an important direction for continuing research.


\begin{thebibliography}{99}

\bibitem{miller-ross}
K. S. Miller, B. Ross, \emph{An Introduction to the Fractional Calculus and Fractional Differential Equations}, John Wiley, New York, 1993.

\bibitem{teodoro-machado-oliveira}
G. S. Teodoro, J. A. Tenreiro Machado, E. C. de Oliveira, ``A review of definitions of fractional derivatives and other operators'', \emph{Journal of Computational Physics} 388 (2019), pp. 195--208.

%\bibitem{baleanu-fernandez}
%D. Baleanu, A. Fernandez, ``On fractional operators and their classifications", \emph{Mathematics} 7(9) (2019), 830.

\bibitem{hilfer-luchko}
R. Hilfer, Y. Luchko, ``Desiderata for fractional derivatives and integrals'', \emph{Mathematics} 7 (2019), 149.

\bibitem{fernandez-fahad}
A. Fernandez, H. M. Fahad, ``On the importance of conjugation relations in fractional calculus'', \emph{Computational and Applied Mathematics} 41(6) (2022), 246.

\bibitem{diethelm}
K. Diethelm, \emph{The analysis of fractional differential equations: An application-oriented exposition using differential operators of Caputo type}, Springer, Heidelberg, 2010.

\bibitem{luchko:level}
Y. Luchko, ``Fractional derivatives and the fundamental theorem of fractional calculus'', \emph{Fractional Calculus and Applied Analysis} 23(4), pp. 939--966.

\bibitem{ross}
B. Ross (ed.), \emph{Fractional Calculus and its Applications: Proceedings of the International Conference Held at the University of New Haven, June 1974}, Springer-Verlag, Berlin, 1975.

\bibitem{cartwright-mcmullen}
D. I. Cartwright, J. R. McMullen, ``A note on the fractional calculus'', \emph{Proceedings of the Edinburgh Mathematical Society} 21 (1978), pp. 79--80.

\bibitem{caolabora}
D. Cao Labora, ``An extension of the Cartwright--McMullen theorem in fractional calculus for the Stieltjes case'', \emph{Journal of Integral Equations and Applications} 37(1) (2025), pp.~1--6.

\bibitem{jornet}
M. Jornet, ``An axiomatic characterization for the multidimensional weighted Riemann--Liouville fractional integral with respect to a function'', \emph{Journal of Integral Equations and Applications} 37(2) (2025), pp.~149--161.

\bibitem{hilfer-kleiner}
R. Hilfer, T. Kleiner, ``Fractional calculus for distributions'', \emph{Fractional Calculus and Applied Analysis} 27(5) (2024), pp. 2063--2123.

\bibitem{kleiner-hilfer}
T. Kleiner, R. Hilfer, ``Convolution on distribution spaces characterized by regularization'', \emph{Mathematische Nachrichten} 296(5), pp. 1938--1963.

\bibitem{rudin}
W. Rudin, \emph{Functional Analysis}, 2nd ed., McGraw-Hill, Singapore, 1991.

\bibitem{simon}
J. Simon, \emph{Banach, Fr\'echet, Hilbert and Neumann Spaces}, ISTE, London, 2017.

\bibitem{fernandez}
A. Fernandez, ``Abstract algebraic construction in fractional calculus: parametrised families with semigroup properties'', \emph{Complex Analysis and Operator Theory} 18 (2024), 50. 

\bibitem{martinez-sanz-martinez}
C. Martinez, M. Sanz, M. D. Martinez, ``About fractional integrals in the space of locally integrable functions'', \emph{Journal of Mathematical Analysis and Applications} 167 (1992), pp. 111--122.

\bibitem{sutherland}
W. A. Sutherland, \emph{Introduction to Metric and Topological Spaces}, 2nd ed., Oxford University Press, Oxford, 2009.

\bibitem{bourbaki}
N. Bourbaki, \emph{General Topology}, Springer-Verlag, Berlin, 1987.

\bibitem{mse}
Mathematics Stack Exchange, ``Simple classification-proof of all continuous homomorphism from the reals (under $+$) to the unit circle group of complex numbers (under $*$)'', 2017: \url{https://math.stackexchange.com/questions/2233122/}

\bibitem{Schwartz}
L. Schwartz, \emph{Th\'eorie des Distributions}, Hermann, Paris, 1966.

\bibitem{folland}
G. B. Folland, \emph{Real Analysis: Modern Techniques and their Applications}, 2nd ed., John Wiley \& Sons, New York, 1999.

\bibitem{royden}
H. L. Royden, P. M. Fitzpatrick, \emph{Real Analysis}, 4th ed., Macmillan, New York, 2010.

\bibitem{beffa}
F. Beffa, \emph{Weakly Nonlinear Systems with Applications in Communications Systems}, Springer, Cham, 2024.

\bibitem{dimovski}
I. H. Dimovski, \emph{Convolutional Calculus}, 2nd ed., Kluwer Academic Publisher, Dordrecht, 1990.

\bibitem{yakubovich-luchko}
S. B. Yakubovich, Y. Luchko, \emph{The hypergeometric approach to integral transforms and convolutions}, Springer Science \& Business Media, Berlin, 1994.

\end{thebibliography}
\end{document}